\documentclass[10pt]{article}

\pdfoutput=1 
\usepackage{etex}
\usepackage[english]{babel}
\usepackage[utf8]{inputenc}
\usepackage[T1]{fontenc}
\usepackage{lmodern}
\usepackage{authblk}

\usepackage{amsmath, latexsym, amssymb, amsfonts, amsthm}
\usepackage[colorlinks,citecolor=black,filecolor=black,linkcolor=black,urlcolor=black ]{hyperref}
\usepackage{graphicx}
\usepackage{subfigure} 
\usepackage{enumitem}
\usepackage{multicol}

\def\ra {\rightarrow}

\def\be{\begin{equation}}   \def\ee{\end{equation}}
\def\ba   {\begin{array}}      \def\ea   {\end{array}}
\def\bea  {\begin{eqnarray}}   \def\eea  {\end{eqnarray}}
\def\bean {\begin{eqnarray*}}  \def\eean {\end{eqnarray*}}

\newtheorem{theorem} {Theorem}
\newtheorem{lemma}{Lemma}
\newtheorem{definition} {Definition}

\newtheorem{remark}{Remark}

\theoremstyle{definition}
\newtheorem{example}{Example}

\newcommand{\bo} {\ensuremath{{\bf i_1}}}
\newcommand{\bos}{\ensuremath{{\bf i_1^{\text 2}}}}
\newcommand{\bts}{\ensuremath{{\bf i_2^{\text 2}}}}

\newcommand {\bjp}{\ensuremath{{\bf j_1}}}

\newcommand {\bjd}{\ensuremath{{\bf j_2}}}

\newcommand{\bt} {\ensuremath{{\bf i_2}}}
\newcommand{\bb} {\ensuremath{{\bf i_3}}}
\newcommand{\bbs} {\ensuremath{{\bf i_3^{\text 2}}}}
\newcommand{\bq} {\ensuremath{{\bf i_4}}}

\newcommand{\bk} {\ensuremath{{\bf i_{k}}}}
\newcommand {\bjt}{\ensuremath{{\bf j_3}}}

\newcommand {\bjk}{\ensuremath{{\bf j_k}}} 

\newcommand{\id}[1]{\gamma_{#1}}
\newcommand{\ic}[1]{\overline{\gamma}_{#1}}

\newcommand{\mB}{\mathbb{B}}
\newcommand{\mC}{\ensuremath{\mathbb{C}}}
\newcommand{\mD}{\ensuremath{\mathbb{D}}}
\newcommand{\mN}{\ensuremath{\mathbb{N}}}

\newcommand{\mR}{\ensuremath{\mathbb{R}}}
\newcommand{\mT}{\ensuremath{\mathbb{T}}}

\newcommand{\cM}{\ensuremath{\mathcal{M}}}
\newcommand{\mM}{\ensuremath{\mathbb{M}}}

\newcommand{\cK}{\ensuremath{\mathcal{K}}}

\newcommand{\op}{\left(}
\newcommand{\fp}{\right)}
\newcommand{\oa}{\left\{}
\newcommand{\fa}{\right\}}
\newcommand{\oc}{\left[}
\newcommand{\fc}{\right]}

\newcommand{\cT}{\mathcal{T}}
\newcommand{\cF}{\mathcal{F}}
\newcommand{\cH}{\mathcal{H}}

\newcommand{\vspan}{\ensuremath{\mathrm{span}}}

\newcommand{\gm}[1]{\ensuremath{\gamma_{#1}}}
\newcommand{\gmc}[1]{\ensuremath{\overline{\gamma}_{#1}}}
\newcommand{\im}[1]{\ensuremath{\mathrm{\mathbf{i}}_{\mathbf{#1}}}}
\newcommand{\jm}[1]{\ensuremath{\mathrm{\mathbf{j}}_{\mathbf{#1}}}}
\newcommand{\TC}{\ensuremath{\mathbb{TC}}}
\newcommand{\mI}{\ensuremath{\mathbb{I}}}
\usepackage{todonotes}
\usepackage{changes}

\title{Tricomplex Distance Estimation for Filled-in Julia Sets and Multibrot Sets}

\author[1]{Guillaume Brouillette \thanks{E-mail: {\tt guillaume.brouillette@uqtr.ca}}}
\author[2]{Pierre-Olivier Paris\'e\thanks{E-mail: {\tt pierre-olivier.parise.1@ulaval.ca}}}
\author[1]{Dominic Rochon\thanks{E-mail: {\tt dominic.rochon@uqtr.ca}}}
\affil[1]{Département de mathématiques et d'informatique, Université du Québec\\
 C.P. 500, Trois-Rivières, Québec, Canada, G9A 5H7\\}
\affil[2]{Département de mathématiques et de statistique, Université Laval\\
1045, av. de la Médecine, Québec, Canada, G1V 0A6.}
\date{}

\begin{document}
\maketitle

\begin{abstract}
In this article, we present a distance estimation formula that can be used to ray trace 3D slices of the filled-in Julia sets and the Multibrot sets generated by the tricomplex polynomials of the form $\eta^p+c$ where $p$ is any integer greater than $1$.

\end{abstract}\vspace{0.5cm} 

\section*{Introduction}

	The complex space may be generalized in a few different ways. A simple but interesting way to extend the complex numbers is to use more imaginary units $\im{}$. Indeed, by using two imaginary units $\im{1}$ and $\im{2}$ rather than a single one as in the complex space, we obtain a four-dimensional number system called the bicomplex numbers. Moreover, introducing a new unit generates two additional dimensions : one associated to the imaginary unit $\im{2}$ and one to the hyperbolic unit $\jm{1}:=\im{1}\im{2}$.

	Furthermore, it is possible to generalize the bicomplex space to the tricomplex space by introducing another imaginary unit $\im{3}$. Again, this new unit doubles the dimension of the space. As a matter of fact, the tricomplex number system contains one real, four imaginary and three hyperbolic units.

	Many concepts and results in complex analysis and dynamics can be generalized to the bicomplex and tricomplex spaces \cite{Luna, RochonMartineau, Rochon1, Rochon3, Rochon2, RochonShapiro, vajiac}. For instance, fractals like the Mandelbrot and Julia sets as well as many of their properties may be extended to the tricomplex space \cite{Dang, GarantPelletier, Martineau, Parise, Wang}. Their generalizations will be the main interest of this article. Since the tricomplex space contains many imaginary and hyperbolic units, fractals become richer when generalized to this space.
	
	To generate visuals of such kind of fractals, we can use the ray tracing technique \cite{Martineau, Peitgen}. To do so, the distance between a point out of the fractal and the fractal itself is needed. Although an exact formula is not known, upper and lower bounds of this distance are known in the complex plane and can be extended to the bicomplex space \cite{Martineau, RochonMartineau}. In this article, we show how to extend and apply these results to the tricomplex space.
	
	 In section 1, we will introduce the tricomplex numbers as well as some related concepts. Then, section 2 presents the Multibrot and the filled-in Julia sets. Some properties are shown and generalized up to the tricomplex space. Next, we see some important results related to the distance from a point to a complex fractal in section 3. Furthermore, we see how these results can be used in when it comes to tricomplex fractals. Finally, in section 4, we see the concept of tridimensional principal slice and use the distance bounds approximations to generate multiple 3D fractals.

\section{Tricomplex Numbers}\label{SecBasics}
A tricomplex number $\eta$ is a pair of bicomplex numbers $\eta_1$ and $\eta_2$ combined with the imaginary unit $\bb$ such that $\bbs = -1$, and
	\begin{align*}
	\eta = \eta_1 + \eta_2 \bb \text{.}
	\end{align*}
Any bicomplex numbers can be written with two complex numbers and an imaginary unit $\bt$ such that $\bts = -1$, and thus
	\begin{align*}
	\eta = \eta_{11} + \eta_{12} \bt + \eta_{21} \bb + \eta_{22} \bjt\text{,}
	\end{align*}
where $\bjt := \bt\bb$, $\eta_1 = \eta_{11} + \eta_{12} \bt$, and $\eta_2 = \eta_{21} + \eta_{22} \bt$. Hence, a tricomplex number $\eta$ may be described by the three following expressions:
	\begin{align*}
	\eta &= \eta_1 + \eta_2 \bb \\
	&= \eta_{11} + \eta_{12} \bt + \eta_{21} \bb + \eta_{22} \bjt \\
	&= x_0 + x_1 \bo + x_2 \bt + x_3 \bjp + x_4 \bb + x_5 \bjd + x_6 \bjt + x_7 \bq\text{,}
	\end{align*}
where $\bjp := \bo \bt$, $\bjd := \bo \bb$, $\bq := \bo \bt \bb$, and $\bos = -1$. The set of tricomplex numbers is denoted, in \cite{Baley}, by $\mC_3$ or by $\mM (3)$ in \cite{GarantRochon, RochonParise}, but we will denote it by $\mT \mC$ in conformity with the set of bicomplex numbers denoted by $\mB \mC$ in many references.

The addition of tricomplex numbers is defined term-by-term, and the product of tricomplex numbers is done by polynomial-like product and using the relations between each units that are gathered in Table \ref{tabC1}.

\begin{table}[h]
\centering
\begin{tabular}{c|*{9}{c}}
$\cdot$ & 1 & $\mathbf{i_1}$ & $\mathbf{i_2}$ & $\mathbf{i_3}$ & $\mathbf{i_4}$ & $\mathbf{j_1}$ & $\mathbf{j_2}$ & $\mathbf{j_3}$\\\hline
1 & 1 & $\mathbf{i_1}$ & $\mathbf{i_2}$ & $\mathbf{i_3}$ & $\mathbf{i_4}$ & $\mathbf{j_1}$ & $\mathbf{j_2}$ & $\mathbf{j_3}$\\
$\mathbf{i_1}$ & $\mathbf{i_1}$ & $-\mathbf{1}$ & $\mathbf{j_1}$ & $\mathbf{j_2}$ & $-\mathbf{j_3}$ & $-\mathbf{i_2}$ & $-\mathbf{i_3}$ & $\mathbf{i_4}$\\
$\mathbf{i_2}$ & $\mathbf{i_2}$ & $\mathbf{j_1}$ & $-\mathbf{1}$ & $\mathbf{j_3}$ & $-\mathbf{j_2}$ & $-\mathbf{i_1}$ & $\mathbf{i_4}$ & $-\mathbf{i_3}$\\
$\mathbf{i_3}$ & $\mathbf{i_3}$ & $\mathbf{j_2}$ & $\mathbf{j_3}$ & $-\mathbf{1}$ & $-\mathbf{j_1}$ & $\mathbf{i_4}$ & $-\mathbf{i_1}$ & $-\mathbf{i_2}$\\
$\mathbf{i_4}$ & $\mathbf{i_4}$ & $-\mathbf{j_3}$  & $-\mathbf{j_2}$ & $-\mathbf{j_1}$ & $-\mathbf{1}$ & $\mathbf{i_3}$ & $\mathbf{i_2}$ & $\mathbf{i_1}$\\
$\mathbf{j_1}$ & $\mathbf{j_1}$ & $-\mathbf{i_2}$  & $-\mathbf{i_1}$ & $\mathbf{i_4}$ & $\mathbf{i_3}$ & $\mathbf{1}$ & $-\mathbf{j_3}$ & $-\mathbf{j_2}$\\
$\mathbf{j_2}$ & $\mathbf{j_2}$ & $-\mathbf{i_3}$  & $\mathbf{i_4}$ & $-\mathbf{i_1}$ & $\mathbf{i_2}$ & $-\mathbf{j_3}$ &  $\mathbf{1}$ & $-\mathbf{j_1}$\\
$\mathbf{j_3}$ & $\mathbf{j_3}$ & $\mathbf{i_4}$ &$-\mathbf{i_3}$  & $-\mathbf{i_2}$ & $\mathbf{i_1}$ & $-\mathbf{j_2}$ & $-\mathbf{j_1}$ &  $\mathbf{1}$ \\
\end{tabular}
\caption{Products of tricomplex imaginary and hyperbolic units.}\label{tabC1}
\end{table}

Hence, $( \TC , + , \cdot )$ forms a commutative unitary ring. Restricted to a vector space, the set of tricomplex numbers may be simply seen as $\mC^4$ or $\mR^8$. The units $\bo$, $\bt$, $\bb$, and $\bq$ are called imaginary units and $\bjp$, $\bjd$, $\bjt$ are called hyperbolic units. Thus, from those units, we define the following subsets of $\mT \mC$, $\mC (\bk ) := \oa x + y \bk \, : \, x, y \in \mR \fa$ for $\bk \in \oa \bo , \bt , \bb , \bq \fa$, and $\mD (\bjk ) := \oa x + y \bjk \, : \, x , y \in \mR \fa$ for $\bjk \in \oa \bjp , \bjd , \bjt \fa$. The sets $\mC (\bk )$ are all isomorphic to the complex numbers, and the sets $\mD (\bjk )$ are all isomorphic to the hyperbolic numbers (see \cite{vajiac2,Sobczyk}), with the usual operations.

It is also important to see that any $\eta \in \mT \mC$ has the following idempotent representation
	\begin{align}
	\eta = (\eta_1 - \eta_2 \bt ) \id{3} + (\eta_1 + \eta_2 \bt ) \ic{3} := \eta_{\id{3}} \id{3} + \eta_{\ic{3}} \ic{3} \label{idempBiDecomp}
	\end{align}
where $\id{3} := \frac{1 + \bjt}{2}$, and $\ic{3} := \frac{1 - \bjt}{2}$. This representation is called idempotent since the tricomplex numbers $\id{3}$ and $\ic{3}$ are idempotent elements, meaning that $\gamma^2 = \gamma$. Moreover, we see that $\id{3} \ic{3} = 0$ and $\id{3} + \ic{3} = 1$. Thus, this representation allows to add, multiply and divide tricomplex numbers term-by-term. If we define the set $S_3 := \oa \id{3} , \ic{3} \fa$\footnote{In \cite[p. 314]{Baley}, $S_n = \oa \id{n-1}, \ic{n-1} \fa$ where $n = 3$ specifically for the set $\mM (3)$.}, then \eqref{idempBiDecomp} can be rewritten as
	\begin{align*}
	\eta = \sum_{\gamma \in S_3} \eta_{\gamma} \gamma \text{.}
	\end{align*}
	There is also another idempotent representation on four idempotent elements. We define $\id{1} := \frac{1 + \bjp}{2}$, and $\ic{1}:= \frac{1 - \bjp}{2}$. Then any $\eta \in \mT \mC$ can be rewritten as
	\begin{align}
		\eta &= \op \oc (\eta_{11} + \eta_{22}) - (\eta_{12} - \eta_{21}) \bo \fc \id{1} + \oc (\eta_{11} + \eta_{22}) + (\eta_{12} - \eta_{21}) \bo \fc \ic{1}  \fp \id{3} \notag\\
		& \quad + \op \oc (\eta_{11} - \eta_{22} ) - (\eta_{12} + \eta_{21}) \bo \fc \id{1} + \oc (\eta_{11} - \eta_{22}) - (\eta_{12} + \eta_{21}) \bo \fc \ic{1} \fp \ic{3} \text{.} \notag\\
		& := \eta_{\id{1}\id{3}} \id{1} \id{3} + \eta_{\ic{1}\id{3}} \ic{1}\id{3} + \eta_{\id{1}\ic{3}} \id{1}\ic{3} + \eta_{\ic{1}\ic{3}} \ic{1}\ic{3} \label{idempTriDecomp}\text{.}
	\end{align}
	The elements $\id{1} \id{3}$, $\ic{1}\id{3}$, $\id{1}\ic{3}$, and $\ic{1}\ic{3}$ are idempotent elements, and $\id{1} \id{3} + \ic{1}\id{3} + \id{1}\ic{3} + \ic{1}\ic{3} = 1$. If we define the set 
		\begin{align*}
		S_{1,3} := \id{3} S_1 \cup \ic{3} S_1 = \oa \id{1} \id{3}, \ic{1} \id{3} , \id{1} \ic{3} , \ic{1} \ic{3} \fa\text{,}
		\end{align*}
	where $S_1 := \oa \id{1} , \ic{1} \fa$, then we may rewrite \eqref{idempTriDecomp} as
		\begin{align*}
		\eta = \sum_{\gamma \in S_{1,3}} \eta_{\gamma} \gamma \text{.}
		\end{align*}
	
	From the idempotent representation, we can define two types of Cartesian products. The first one is called the $\mT \mC$-Cartesian product, and is defined as
	\begin{align*}
	X_1 \times_{\id{3}} X_2 := \oa x_1 \id{3} + x_2 \ic{3} \, : \, x_1 \in X_1 \text{, } x_2 \in X_2 \fa
	\end{align*}			
	where $X_1, X_2 \subset \mB \mC$. It is easy to see that the application\label{eqGamma2}
		\begin{align*}
		\Gamma_2 : X_1 \times X_2 &\ra X_1 \times_{\id{3}} X_2\\
		 (x_1,x_2 )&\mapsto x_1 \id{3} + x_2 \ic{3}
		\end{align*}
	is a homeomorphism. The second one is called the $\mB\mC$-Cartesian product, and is defined as
	\begin{align*}
	X_1 \times_{\id{1}} X_2 := \oa x_1 \id{1} + x_2 \ic{1} \, : \, x_1 \in X_1 \text{, } x_2 \in X_2 \fa
	\end{align*}
	where $X_1 , X_2 \subset \mC (\bo )$. It is also easy to see that the application\label{eqGamma1}
		\begin{align*}
		\Gamma_1 : X_1 \times X_2 &\ra X_1 \times_{\id{1}} X_2\\
		(x_1, x_2) &\mapsto x_1 \id{1} + x_2 \ic{1}
		\end{align*}
	is a homeomorphism. We  respectively denote by $\id{3} : \mT \mC \ra \mB \mC$ and $\ic{3} : \mT \mC \ra \mB \mC$ the projections on the first and the second idempotent component of the idempotent representation of a tricomplex number in $\id{3}$ and $\ic{3}$. Similarly, we denote respectively by $\id{1} : \mB \mC \ra \mC (\bo )$ and $\ic{1} : \mB \mC \ra \mC (\bo )$ the projections on the first and the second idempotent component of the idempotent representation of a bicomplex number in $\id{1}$ and $\ic{1}$.
	
	In this article, we use the Euclidean norm (noted $\Vert \cdot \Vert_3$) of $\mR^8$. However, from \cite{Baley}, we have the following formula for the norm
		\begin{align}\label{NormIdem}
		\Vert \eta \Vert_3 &= \sqrt{\frac{\Vert \eta_{\id{3}} \Vert_2^2 + \Vert \eta_{\ic{3}} \Vert_2^2}{2}} = \sqrt{\frac{|\eta_{\id{1}\id{3}}|^2 + |\eta_{\ic{1}\id{3}}|^2 + |\eta_{\id{1}\ic{3}}|^2 + |\eta_{\ic{1}\ic{3}}|^2}{4}}
		\end{align}
where $\Vert \cdot \Vert_2$ is the Euclidean norm of a bicomplex number, and $|\cdot |$ is the Euclidean norm of a complex number. 

	There are several important types of discus in the tricomplex space. Let $r_1 , r_2 > 0$, and $\eta \in \mT\mC$. We define the \textbf{tricomplex open discus} $D_3(\eta, r_1, r_2)$ as the $\mT\mC$-Cartesian product of two balls $B_2 (\eta_{\id{3}} , r_1)$, and $B_2( \eta_{\ic{3}} , r_2)$ of $\mM (2)$, that is $D_3 (\eta , r_1 , r_2 ) := B_2 (\eta_{\id{3}} , r_1) \times_{\id{3}} B_2 ( \eta_{\ic{3}} , r_2 )$. In a similar way, we define the \textbf{tricomplex closed discus} as $\overline{D_3}( \eta , r_1 , r_2 ) := \overline{B_2}(\eta_{\id{3}} , r_1) \times_{\id{3}} \overline{B_2} (\eta_{\ic{3}} , r_2)$. If $r:= r_1 = r_2$, then we simply denote the opened, and closed discii as $D_3 (\eta , r)$, and $\overline{D_3} (\eta , r)$ respectively. Similar definitions for the opened, and closed discii hold for the bicomplex case.

	In the next section, we will introduce the filled-in Julia sets, and the Multibrot sets of a certain type of polynomials.

\section{Filled-in Julia Sets and Multibrot Sets}\label{sectionJuliaProperties}
	In this section, we define the filled-in Julia sets and Multibrot sets for the family of polynomials $f_c(z) := z^p + c$ where $p \geq 2$ is an integer, and $c$ is a complex, bicomplex or tricomplex number.

	In the case of complex numbers, this family characterized all the polynomials of degree $p$ which have exactly one critical point. In fact, according to J. Milnor \cite{Milnor1}, these polynomials are conjugated by an affine transformation to the polynomial $f_c$, where the parameter $c$ is unique up to a multiplication by a $(p-1)$-th root of unity. The reader may also see \cite{Parise} for details.

	We start by the case where $c$ is a complex number. We suppose that $c \in \mC (\bo )$ for the sake of consistency. 
	
	The \textbf{filled-in Julia} set of the polynomial $f_c (z) = z^p + c$ for fixed integer $p \geq 2$ and complex number $c \in \mC (\bo )$ is defined as the set
		\begin{align*}
		\cK_{c}^p := \oa z \in \mC (\bo ) \, : \, \oa f_c^m (z) \fa_{m = 1}^{\infty} \text{ is bounded} \fa \text{.}
		\end{align*}
	Here are some basic properties related with these sets \cite{Milnor2}:
		\begin{enumerate}
		\item $\cK_{c}^p \subset \overline{B_1} (0, R)$ where $R := \max \oa |c|, 2^{1/(p-1)} \fa$;
		\item $z \in \cK_{c}^p$ if and only if $|f_c^m (z) | \leq R$ for all integer $m \geq 1$;
		\item $\cK_{c}^p$ is a closed set, and it is a connected set if and only if $0 \in \cK_{c}^p$.
		\end{enumerate}
	
	The last property is a way to introduce the \textbf{Multibrots} denoted by $\cM^p$. There are defined as
		\begin{align*}
		\cM^p &:= \oa c \in \mC (\bo ) \, : \, \cK_{c}^p \text{ is connected} \fa \\
		&= \oa c \in \mC (\bo ) \, : \, \oa f_c^m (0) \fa_{m=1}^{\infty} \text{ is bounded} \fa \text{.}
		\end{align*}
	In \cite{RochonParise}, it is shown that 
		\begin{enumerate}
		\item $\cM^p \subset \overline{B_1} (0, 2^{1/(p-1)})$;
		\item $c \in \cM^p$ if and only if $|f_c^m(0)| \leq 2^{1/(p-1)}$ $\forall m \geq 1$;
		\item $\cM^p$ is a closed set.
		\end{enumerate}
	
	When the parameter $c \in \mB \mC$ and fixed, we define the \textbf{bicomplex filled-in Julia} set of the polynomial $f_c$ for a fixed integer $p \geq 2$ as the set
		\begin{align*}
		\cK_{2,c}^p := \oa w \in \mB \mC \, : \, \oa f_c^m (w ) \fa_{m = 1}^{\infty} \text{ is bounded} \fa \text{.}
		\end{align*}
	The sets $\cK_{2,c}^p$ has the following properties.
		\begin{theorem}\label{thmPropK2}
		Let $p \geq 2$ be an integer, and $c \in \mB \mC$. Then,
			\begin{enumerate}
			\item $\cK_{2,c}^p = \cK_{c_{\id{1}}}^p \times_{\id{1}} \cK_{c_{\ic{1}}}^p$ where $c = c_{\id{1}} \id{1} + c_{\ic{1}} \ic{1}$;
			\item $\cK_{2, c}^p \subset \overline{D_{2}} (0, R_{\gm{1}}, R_{\gmc{1}})$ where $R_{\gamma} := \max \oa |c_{\gamma}|, 2^{1/(p-1)} \fa$ for $\gamma \in S_1$;
			\item $\cK_{2,c}^p$ is a compact set of $\mB \mC$.
			\end{enumerate}
		\end{theorem}
		\begin{proof}
		Each part follows from the fact that the iterates of $f_c$ can be expressed as $f_c^m (w) = \sum_{\gamma \in S_1} f_{c_{\gamma}}^m (w_{\gamma}) \gamma$, and so
			\begin{align*}
			\Vert f_c^m (w) \Vert_2 = \sqrt{\frac{|f_{c_{\id{1}}} (w_{\id{1}}) |^2 + |f_{c_{\ic{1}}} (w_{\ic{1}} )|^2}{2}} \text{.}
			\end{align*}
		Thus, the sequence $\oa f_{c}^m (w) \fa_{m=1}^{\infty}$ remains bounded if and only if the sequence $\oa f_{c_{\gamma}}^m (w_{\gamma})\fa_{m=1}^{\infty}$ is bounded $\forall \gamma \in S_1$, and we obtain the first statement. Considering that $\cK_{c_\gamma}^p \subset \overline{B_1}(0,R_\gamma)$ where $R_\gamma = \max\{|c_\gamma|, 2^{1/(p-1)}\}$, the first statement implies the second. Finally, the application $\Gamma_1$ introduced in section \ref{eqGamma1} is a homeomorphism between $\cK_{2,c}^p$ and the compact set $\cK_{c_{\gm{1}}}^p \times \cK_{c_{\gmc{1}}}^p$. Consequently, the set $\cK_{2,c}^p$ is also compact.
		\end{proof}
		\begin{theorem}\label{thmConnectedofM2}
		$\cK_{2,c}^p$ is a connected set if and only if $0 \in \cK_{2, c}^p$.
		\end{theorem}
		\begin{proof}
		The case $p = 2$ of this theorem appears in \cite{Rochon1, Rochon2}. More generally, let $c = c_{\gm{1}}\gm{1} + c_{\gmc{1}}\gmc{1}$. Since $\Gamma_1$, defined in section \ref{eqGamma1}, is a homeomorphism, we have that $\cK_{2,c}^p$ is a connected set if and only if $\cK_{c_1}^p$ and $\cK_{c_2}^p$ are connected sets in the complex plane. Therefore, we conclude that $\cK^p_{2,c}$ is connected if and only if $0\in\cK^p_{c_{\gm{1}}}$ and $0\in\cK^p_{c_{\gmc{1}}}$.
		\end{proof}
	
	Now, according to Theorem \ref{thmConnectedofM2}, we define the \textbf{bicomplex Multibrot} set of order $p$, denoted by $\cM_2^p$ for a fixed integer $p \geq 2$, as the set
		\begin{align*}
		\cM_2^p & := \oa c \in \mB \mC \, : \, \cK_{2,c}^p \text{ is connected} \fa \\
		&= \oa c \in \mB \mC \, : \, \oa f_c^m (0) \fa_{m = 1}^{\infty} \text{ is bounded} \fa \text{.}
		\end{align*}
	In \cite{RochonParise, Wang}, it is proved that the set $\cM_2^p$ of order $p \geq 2$ has the following properties:
	\begin{enumerate}
		\item $\cM_2^p = \cM^p \times_{\id{1}} \cM^p$;
		\item $\cM_2^p \subset \overline{D_2} (0, 2^{1/(p-1)} )$;
		\item $c \in \cM_2^p$ if and only if $\Vert f_c^m (0) \Vert_2 \leq 2^{1/(p-1)}$ $\forall m \geq 1$;
		\item $\cM_2^p$ is a compact and connected subset of $\mB \mC$.
	\end{enumerate}
	
	Finally, let $c \in \mT \mC$, and $p \geq 2$ be an integer. The \textbf{tricomplex filled-in Julia set} of order $p$ is defined as 
	\begin{align*}
		\cK_{3,c}^p := \oa \eta \in \mT \mC \, : \, \oa f_c^m (\eta ) \fa_{m = 1}^{\infty} \text{ is bounded} \fa \text{.}
	\end{align*}
	
	The set $\cK_{3,c}^p$ has the following useful properties.
		\begin{theorem}
		Let $p \geq 2$ be an integer, and $c \in \TC$ with $c = \sum_{\gamma \in S_{1,3}} c_{\gamma} \gamma$. Then
			\begin{enumerate}
			\item $\cK_{3,c}^p = \op \cK_{c_{\id{1}\id{3}}}^p \times_{\id{1}} \cK_{c_{\ic{1}\id{3}}}^p \fp \times_{\id{3}} \op \cK_{c_{\id{1}\ic{3}}}^p \times_{\id{1}} \cK_{c_{\ic{1}\ic{3}}}^p \fp$;
			\item $\cK_{3,c}^p \subset \overline{D_2} (0, R_{\id{1}\id{3}}, R_{\ic{1}\id{3}} ) \times_{\id{3}} \overline{D_2}(0, R_{\id{1}\ic{3}}, R_{\ic{1}\ic{3}})$ where \\$R_{\gamma} := \max \oa |c_{\gamma}|, 2^{1/(p-1)} \fa$ for $\gamma \in S_{1,3}$;
			\item $\cK_{3,c}^p$ is a compact subset of $\mT \mC$.
			\end{enumerate}
		\end{theorem}
		\begin{proof}
		These properties are a consequence of the idempotent representation of the iterates of $f_c$. The complete proof is similar to the one of Theorem \ref{thmPropK2}.
		\end{proof}
	Similarly to the set $\cK_{2,c}^p$, we have the following characterization.
		\begin{theorem}
		$\cK_{3,c}^p$ is connected if and only if $0 \in \cK_{3,c}^p$.
		\end{theorem}
		
		\begin{proof}
		The case $p = 2$ is shown in \cite{GarantRochon}. Furthermore, when $p > 2$, the approach is similar to the proof of Theorem \ref{thmConnectedofM2}.
		\end{proof}
	
	Following that, we define a \textbf{tricomplex Multibrot} set $\cM_3^p$ of order $p \geq 2$ as the set 
		\begin{align*}
		\cM_3^p := \oa c \in \mT \mC \, : \, \oa f_c^m (0) \fa_{m = 1}^{\infty} \text{ is bounded} \fa \text{.}
		\end{align*}
	Similarly, we have the following basic properties: 
	\begin{enumerate}
		\item $\cM_3^p = \cM_2^p \times_{\id{3}} \cM_2^p$;
		\item $\cM_3^p \subset \overline{D_2}(0, 2^{1/(p-1)}) \times_{\id{3}} \overline{D_2} (0, 2^{1/(p-1)})$;
		\item $c \in \cM_3^p$ if and only if $\Vert f_c^m(0) \Vert_3 \leq 2^{1/(p-1)}$ $\forall m \geq 1$;
		\item $\cM_3^p$ is a compact and connected subset of $\mT \mC$.
	\end{enumerate}
	The proofs of 1 and 2 as well as the proof of the connectedness of $\cM_3^p$ can be found in \cite{Parise, RochonParise}. The third property is a direct consequence of the second. Furthermore, the application $\Gamma_2$ introduced in section \ref{eqGamma2} forms a homeomorphism between $\cM_2^p \times \cM_2^p$ and $\cM_3^p$. Therefore, since $\cM_2^p \times \cM_2^p$ is a compact set, it can be concluded that $\cM_3^p$ is also compact.
	
\section{Distance Estimator}
	It is well-known (see, for example, \cite{Dang} or \cite{Peitgen}) that the distance from a point in the complement of the classical filled-in Julia sets $\cK_{c}^2$ or in the complement of the classical Mandelbrot set $\cM^2$ to the set itself can be estimated. Our purpose is now to extend that result to the sets $\cK_{c}^p$ and $\cM^p$ for any integer $p \geq 2$. Moreover, we will extend the results from É. Martineau and D. Rochon \cite{RochonMartineau} to the tricomplex filled-in Julia sets and Multibrots.
	
	\subsection{Green's functions}
	
	In J. Milnor's book \cite{Milnor2}, it is showed that for any complex-valued polynomial $g(z) = \sum_{k = 1}^n a_k z^k$ where $a_n = 1$ and $n \geq 2$, there exists a unique biholomorphic function (called the Böttcher coordinates) $\phi$ which conjugates $g$ to the application $z \mapsto z^n$ on a certain neighborhood of infinity. More precisely, $\phi$ has the following properties:
		\begin{enumerate}
		\item $\phi (z) \sim z$ as $|z| \ra \infty$; 
		\item $\phi (g(z)) = (\phi (z))^p$;
		\item $\phi (z) = \displaystyle \lim_{m \ra \infty} \oc g^m (z)\fc^{1/p^m}$.
		\end{enumerate}
	Moreover, if the filled-in Julia set associated to $g$ is connected (that is, the filled-in Julia set contains every zeros of $g$ \cite{Milnor2}), then $\phi$ extends to a conformal isomorphism from the complement of the filled-in Julia set to the complement of the closed unit disk. Then, for the polynomial $f_c(z) = z^p + c$ (with $p \geq 2$), there exists a biholomorphic function $\phi_c : \mC (\bo ) \backslash \cK_{c}^p \ra \mC (\bo ) \backslash \overline{B_1} (0, 1)$ where $\cK_{c}^p$ is a connected set.
	
	In fact, A. F. Beardon \cite{Beardon} proved that there exists a biholomorphic mapping $\psi:\mC(\im{1})\backslash\cM^2\rightarrow\mC(\im{1})\backslash\overline{B_1}(0,1)$, used by Douady and Hubbard in some of their work \cite{Douady}, with the same properties as $\phi_c$ which conjugates the application $c \mapsto c^2 + c$ to $c \mapsto c^2$. Moreover, the proof can be easily adapted for the Multibrot sets $\cM^p$ of the complex plane. Thus, we can say that there exists a biholomorphic mapping $\psi:\mC(\im{1})\backslash\cM^p\rightarrow\mC(\im{1})\backslash\overline{B_1}(0,1)$ which has the same properties as $\phi_c$ which conjugates $c \mapsto f_c(c) = f_c^2(0)$ to $c \mapsto c^p$. Loosely speaking, $\psi$ is defined as $\psi (c) := \phi_c(c)$ on $\mC (\bo ) \backslash \cM^p$ for any integer $p \geq 2$.
	
	In the complex plane, with the function $\phi_c : \mC (\bo ) \backslash \cK_c^p \ra \mC (\bo ) \backslash \overline{B_1}(0,1)$, we define the \textbf{Green's function} of the filled-in Julia set $\cK_{c}^p$ as
		\begin{align*}
		G(z) := \oa 
			\begin{matrix}
			0 & \text{ if } z \in \cK_c^p \\
			\ln |\phi_c(z) | & \text{ if } z \not \in \cK_c^p \text{.}
			\end{matrix} \right.
		\end{align*}
	Using the properties of the function $\phi_c$ (see \cite{Milnor2}), it is easy to show that: 
		\begin{enumerate}
			\item $G$ is harmonic, meaning that
			\[\frac{\partial^2 G}{\partial x^2} + \frac{\partial^2 G}{\partial y^2} = 0\quad\text{where } z = x + y\im{1};\]
			\item $G(z) = \displaystyle\lim_{m \ra \infty} \frac{\ln |f_c^m (z)|}{p^m}$ on $\mC (\bo ) \backslash \cK_c^p$ (the convergence is uniform on the compact subsets);
			\item $G(z) \ra 0$ as $z \ra \partial \cK_c^p$;
			\item $G(z) \sim \ln |z|$ as $|z| \ra \infty$.
		\end{enumerate}
	Similarly, with the application $\psi : \mC (\bo ) \backslash \cM^p \ra \mC (\bo ) \backslash \overline{B_1}(0,1)$, we define the Green's function of the multribrot set $\cM^p$ as
		\begin{align*}
		G(c) := \oa 
			\begin{matrix}
			0 & \text{ if } c \in \cM^p \\
			\ln |\psi (c) | & \text{ if } c \not \in \cM^p \text{.}
			\end{matrix} \right.
		\end{align*}
		
		Moreover, in each case, we define the derivative of $G$ as follows:
		\begin{align*}
			G'(z) = \left(\frac{\partial G}{\partial x}, \frac{\partial G}{\partial y}\right).
		\end{align*}
		This definition of $G'$, viewed as the derivative of a mapping of two real variables, will prove itself useful in subsection \ref{FundamentalComplex}.

	\subsection{Fundamental results in the complex plane}\label{FundamentalComplex}
	
	\begin{theorem}[Koebe's 1/4] (See \cite{Conway, Dang}.)
	Let $f$ be an univalent and analytic function on the open disc $B:=B_1 (0,1)$ such that $f(0) = z_0$ and $|f'(0)| = R$. Then, $B_1 (z_0,\frac{R}{4} ) \subset f (B)$.
	\end{theorem}
	
	This classical result leads to the following theorem:
	\begin{theorem}\label{ComplexJBoundsDist}
	Let $\cK_{c}^p$ be connected and $z_0 \in \mC(\im{1})\backslash\cK_c^p$. Define $d(z_0, \cK_{c}^p) := \inf\oa |z - z_0|\, : \, z \in \cK_c^p \fa$. Then,
		\begin{align*}
		\frac{\sinh (G(z_0))}{2e^{G(z_0)}|G'(z_0)|} < d(z_0, \cK_c^p ) < \frac{2 \sinh (G(z_0))}{|G'(z_0)|}
		\end{align*}	
	where $G$ is the Green's function of the set $\cK_c^p$.
	\end{theorem}
	\begin{proof}
	The proof presented in \cite{Dang, Martineau} for $p = 2$ remains valid when $p > 2$. Thus, we present here a summarized version of this proof.
		
	We start with the upper bound. Let $W : \mC (\bo ) \backslash \cK_{c}^p \ra \mC (\bo )$ be defined as $W(z) := \frac{1}{\phi_c(z)}$ with $z \in \mC (\bo ) \backslash \cK_c^p$. It is clear that $|W(z)| < 1$ $\forall z \in \mC (\bo ) \backslash \cK_c^p$. Let $F (w) := \frac{w - w_0}{1 - w \overline{w_0}}$ where $w_0 := W(z_0)$. Then, $F$ is a Möbius transformation since $1 - w_0 \overline{w_0} \neq 0$. Its derivative is
		\begin{align*}
		F'(w) = \frac{1 - |w_0|^2}{(1 - w\overline{w_0})^2}
		\end{align*}
	and so $F'(w_0) = \frac{1}{1 - |w_0|^2} >0$ since $W (z_0) \in B_1 (0, 1)$. Moreover, a little computation shows that $|w| < 1$ if and only if $|F(w)| < 1$. Thus, we have that $\left(F \circ W\right) (B_1 (z_0, r)) \subset B_1 (0,1)$ for any ball $B_1 (z_0 , r) \subset \mC (\bo ) \backslash \cK_c^p$.
	
	Now, let $r_0 := d(z_0, \cK_c^p)$, and from the Schwarz's Lemma \cite{Rudin}, we have 
		\begin{align*}
		|(F\circ W)'(z_0)| < \frac{1}{r_0}
		\end{align*}
	which implies that
		\begin{align*}
		r_0 < \frac{1 - w_0 \overline{w_0}}{|W'(z_0)|} \text{.}
		\end{align*}
	It can be shown that $|W'(z_0)| = e^{-G(z_0)}|G'(z_0)|$ and 
		\begin{align*}
		1 - w_0 \overline{w_0} = e^{-G(z_0)}2 \sinh (G(z_0))\text{.}
		\end{align*}
	Therefore,
		\begin{align*}
		d(z_0, \cK_c^p) = r_0 < \frac{1 - w_0 \overline{w_0}}{|W'(z_0)|} = \frac{2\sinh (G(z_0))}{|G'(z_0)|} \text{.}
		\end{align*}
		
		Now, for the lower bound, in addition to the previous notations, we define another transformation $\omega$ as
			\begin{align*}
			\omega (z) := (F \circ W)^{-1} (|w_0| z) \text{.}
			\end{align*}
		This function is univalent on $B_1(0,1)$ and it can be computed that $\omega(0) = z_0$. Thus, from the Koebe 1/4 Theorem, it follows that
			\begin{align*}
			B_1 \left(z_0, \frac{R}{4} \right) \subset \omega (B_1 (0,1)) \subset \mC (\bo ) \backslash \cK_c^p
			\end{align*}
		where $R := |\omega'(0)|$. It can be shown that
		\[R = \frac{2\sinh (G(z_0))}{e^{G(z_0)}|G'(z_0)|},\] and so
			\begin{align*}
			d(z_0, \cK_c^p) = r_0 > \frac{R}{4} = \frac{\sinh (G(z_0))}{2e^{G(z_0)}|G'(z_0)|}.
			\end{align*}
	\end{proof}
	
	\begin{lemma}\label{Gapprox}
	Let $G$ be the Green's function defined on the set $\mC(\im{1})\backslash\cK_c^p$. We have, for a sufficiently large $m$,
		\begin{align*}
		G(z)\approx \frac{\ln|z_m|}{p^m} \quad\text{and}\quad|G'(z)| \approx \frac{|z_m'|}{p^m |z_m|}
		\end{align*}
	where $z_m := f_c^m(z)$, and $z_m' := \frac{d}{dz} (f_c^m (z))$.
	\end{lemma}
	\begin{proof}
	From the properties of $G(z)$, we know directly that
	\[G(z) = \lim_{m\rightarrow\infty}\frac{\ln|z_m|}{p^m},\]
	where the convergence is uniform on the compact subsets of $\mC(\im{1})\backslash\cK_c^p$.
	Hence, we can see that
		\begin{align*}
		|G'(z)| &\approx \left|\left(\frac{\partial}{\partial x}\left(\frac{\ln |z_m|}{p^m} \right), \frac{\partial}{\partial y}\left(\frac{\ln |z_m|}{p^m} \right)\right)\right|\\
		&= \left|\left(\frac{\frac{\partial}{\partial x}|z_m|}{p^m|z_m|}, \frac{\frac{\partial}{\partial y}|z_m|}{p^m|z_m|}\right)\right|\\
		&= \frac{1}{p^m|z_m|}\left|\left(\frac{\partial |z_m|}{\partial x}, \frac{\partial |z_m|}{\partial y}\right)\right|\\
		&= \frac{||z_m|'|}{p^m |z_m|} \\
		&= \frac{|z_m'|}{p^m |z_m|}
		\end{align*}
	since, for any holomorphic function $f$ on an open set $U \subset \mC (\bo )$, we have $|f'(z)| = ||f(z)|'|$ $\forall z \in U$ (see appendix B in \cite{Martineau}).
	\end{proof}
	
	This lemma leads to the following approximation theorem for the bounds of the distance $d(z_0, \cK_c^p)$.	
	\begin{theorem}\label{ComplexJDistEstimate}
	Let $z_0$ be a point in $\mC(\im{1})\backslash \cK_c^p$ close to the set $\cK_c^p$. The distance $d(z_0, \cK_c^p)$ from the point $z_0$ to the set $\cK_c^p$ can be approximated in the following way:
		\begin{align*}
		\frac{|z_m|\ln |z_m|}{2|z_m|^{1/p^m} |z_m'|} < d(z_0, \cK_c^p) < \frac{2|z_m|\ln |z_m|}{|z_m'|} \text{.}
		\end{align*}
	\end{theorem}
	\begin{proof}
	It is known that $\sinh(x)\approx x$ when $x\ra 0$. Thus, from Lemma \ref{Gapprox} and since $G(z) \ra 0$ when $z \ra \partial \cK_c^p$, we have that the upper bound can be approximated by
		\begin{align*}
		\frac{2 \sinh (G(z_0))}{|G'(z_0)|} \approx \frac{2 |G(z_0)|}{|G'(z_0)|} \approx 2 \frac{\ln |z_m|}{p^m} \frac{p^m |z_m|}{|z_m'|} = \frac{2 |z_m| \ln |z_m|}{|z_m'|} \text{.}
		\end{align*}
	
	For the lower bound, we have that $e^{G(z_0)} \approx |z_m|^{1/p^m}$ when $m\ra \infty$, and so
		\begin{align*}
		\frac{\sinh (G(z_0))}{2e^{G(z_0)} |G'(z_0)|} &\approx \frac{G(z_0)}{2e^{G(z_0)} |G'(z_0)|} \approx \frac{|z_m|\ln |z_m|}{2|z_m|^{1/p^m}|z_m'|}
		\end{align*}
		for $m$ great enough.
	\end{proof}
	
	\begin{figure}
		\centering
		\subfigure[$\cK^2_{-1+0.2\im{1}}$]{\includegraphics[width=0.3\linewidth]{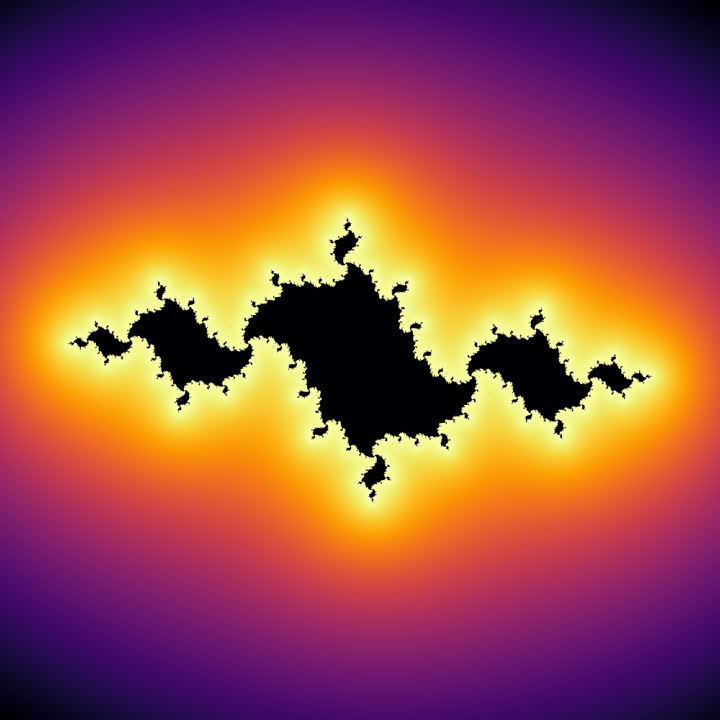}}
		\subfigure[$\cK^3_{0.1+0.8\im{1}}$]{\includegraphics[width=0.3\linewidth]{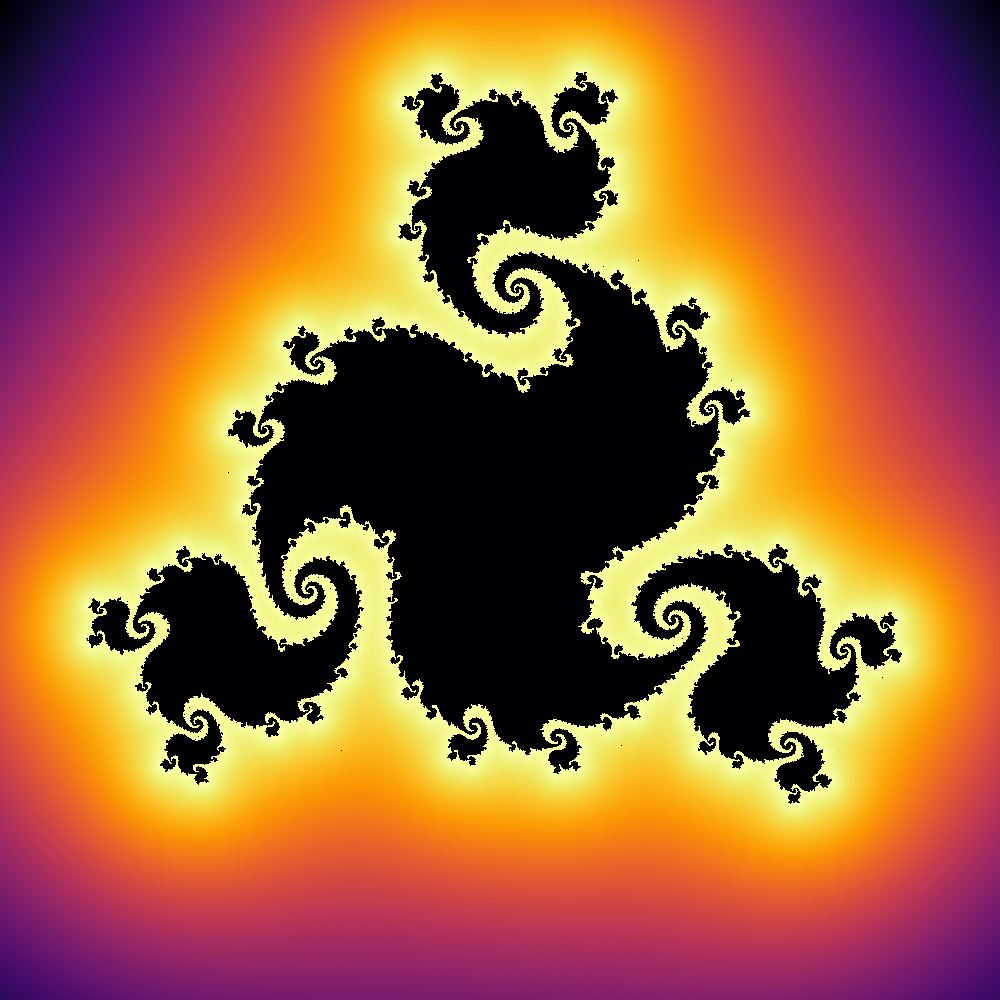}}
		\subfigure[$\cK^4_{-0.7+0.3\im{1}}$]{\includegraphics[width=0.3\linewidth]{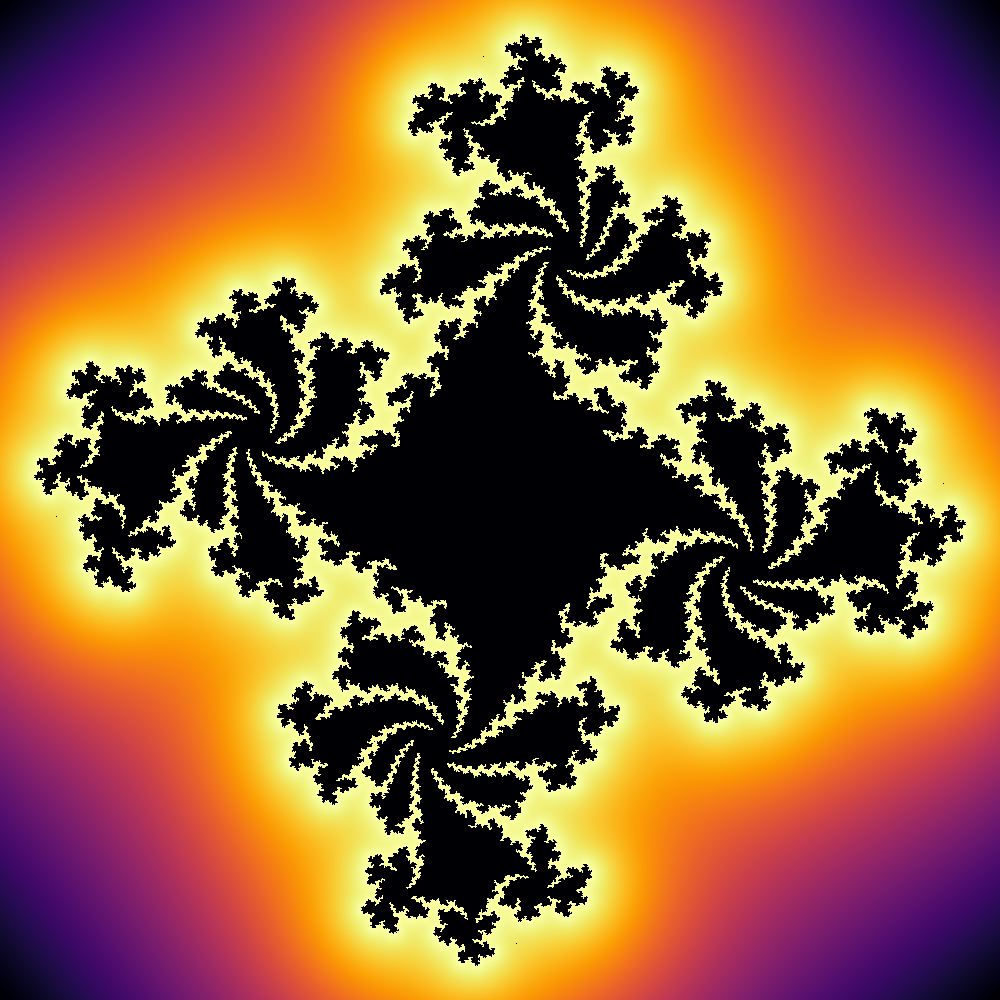}}
		\caption{Various filled-in Julia sets.}
	\end{figure}
	
	With similar arguments with $\psi(c):=\phi_c(c)$, we have the following theorem for the bounds of the distance from a point $c_0 \in \mC(\im{1})\backslash\cM^p$ to the set $\cM^p$.
		\begin{theorem}
		Let $c_0 \in \mC(\im{1})\backslash\cM^p$ and define 
			\begin{align*}
			d (c_0, \cM^p) := \inf \oa |c - c_0| \, : \, c \in \cM^p \fa\text{.}
			\end{align*}
		Then,
			\begin{align*}
			\frac{\sinh (G(c_0))}{2e^{G(c_0)}|G'(c_0)|} < d (c_0 , \cM^p ) < \frac{2\sinh (G(c_0))}{|G'(c_0)|}
			\end{align*}
		where $G$ is the Green's function of the set $\cM^p$.
		\end{theorem}
	An approximation of those last bounds may be found explicitly. First, we need this next lemma.
		\begin{lemma}\label{GApproxMandel}
		Let $G$ be the Green's function defined on the set $\mC(\im{1})\backslash\cM^p$. We have, for a sufficiently large $m$,
		\begin{align*}
		G(c)\approx \frac{\ln|c_m|}{p^{m-1}}, \quad |G'(c)| \approx \frac{|c_m'|}{p^{m-1} |c_m|}\quad\text{and}\quad e^{G(c)} \approx |c_m|^{\frac{1}{p^{m-1}}}
		\end{align*}
		where $c_m := f_c^m(0)$, and $c_m' := \frac{d}{dc} (f_c^m (0))$.
		\end{lemma}
		\begin{proof}
		We know that $f^m_c(c) = f^{m+1}_c(0)$ for all $m\in\mN$. Let $G_{\cM^p}$ and $G_{\cK_c^p}$ be the Green's functions defined on the sets $\mC(\im{1})\backslash\cM^p$ and $\mC(\im{1})\backslash\cK_c^p$ respectively. Since $\psi (c) = \phi_c (c)$ and using the properties of $\phi_c$ for $c \not \in \cM^p$, we can see that $G_{\cM^p}$ can be expressed as
			\begin{align*}
			G_{\cM^p}(c) &= \ln |\phi_c(c)| = \lim_{m \rightarrow \infty} \frac{\ln \left| f_c^m(c) \right|}{p^m} = p \lim_{m \rightarrow \infty} \frac{\ln \left|f_c^{m+1}(0) \right|}{p^{m+1}} = p\cdot G_{\cK_c^p}(0).
			\end{align*}
		Hence, for a sufficiently large $m$, we have from Lemma \ref{Gapprox} that
			\begin{align*}
			G_{\cM^p}(c) = p\cdot G_{\cK_c^p}(0) \approx p\cdot\frac{\ln|f^m_c(0)|}{p^m} = \frac{\ln|c_m|}{p^{m-1}}.
			\end{align*}
		It follows that $e^{G_{\cM^p}(c)} \approx |f^m_c(0)|^{\frac{1}{p^{m-1}}}$ when $m \ra \infty$. Moreover, since $||f(c)|'| = |f'(c)|$ $\forall c \in U$ for any holomorphic function $f$ on an open set $U \subset \mC (\bo )$,
			\begin{align*}
			\left|G'_{\cM^p}(c)\right| &\approx \left|\left(\frac{\partial}{\partial x}\left(\frac{\ln|c_m|}{p^{m-1}}\right), \frac{\partial}{\partial y}\left(\frac{\ln|c_m|}{p^{m-1}}\right)\right)\right| \\
			&= \frac{1}{p^{m-1}|c_m|}\left|\left(\frac{\partial |c_m|}{\partial x}, \frac{\partial |c_m|}{\partial y}\right)\right|\\
			&= \frac{||c_m|'|}{p^{m-1}|c_m|}\\
			&= \frac{|c_m'|}{p^{m-1}|c_m|}.\qedhere
			\end{align*}
		\end{proof}
	Hence, using Lemma \ref{GApproxMandel}, approximations for the bounds of the distance from a point to a Multibrot may be proven and found explicitly. The approach is similar to the proof of Theorem \ref{ComplexJDistEstimate}.
	\begin{theorem}\label{ComplexMDistEstimate}
		The distance $d(c_0, \cM^p)$ from a point $c_0\in \mC(\im{1})\backslash\cM^p$ to the set $\cM^p$ can be approximated in the following way:	
		\begin{align*}
			\frac{|c_m|\ln|c_m|}{2|c_m|^{1/p^{m-1}}|c_m'|} < d (c_0 , \cM^p ) < \frac{2|c_m|\ln|c_m|}{|c_m'|}
		\end{align*}
		where $c_m := f_{c_0}^m(0)$, and $c_m' := \frac{d}{dc} \left.(f_c^m (0))\right|_{c=c_0}$.
	\end{theorem}
	
	\begin{figure}
		\centering
		\subfigure[$\cM^3$]{\includegraphics[width=0.3\linewidth]{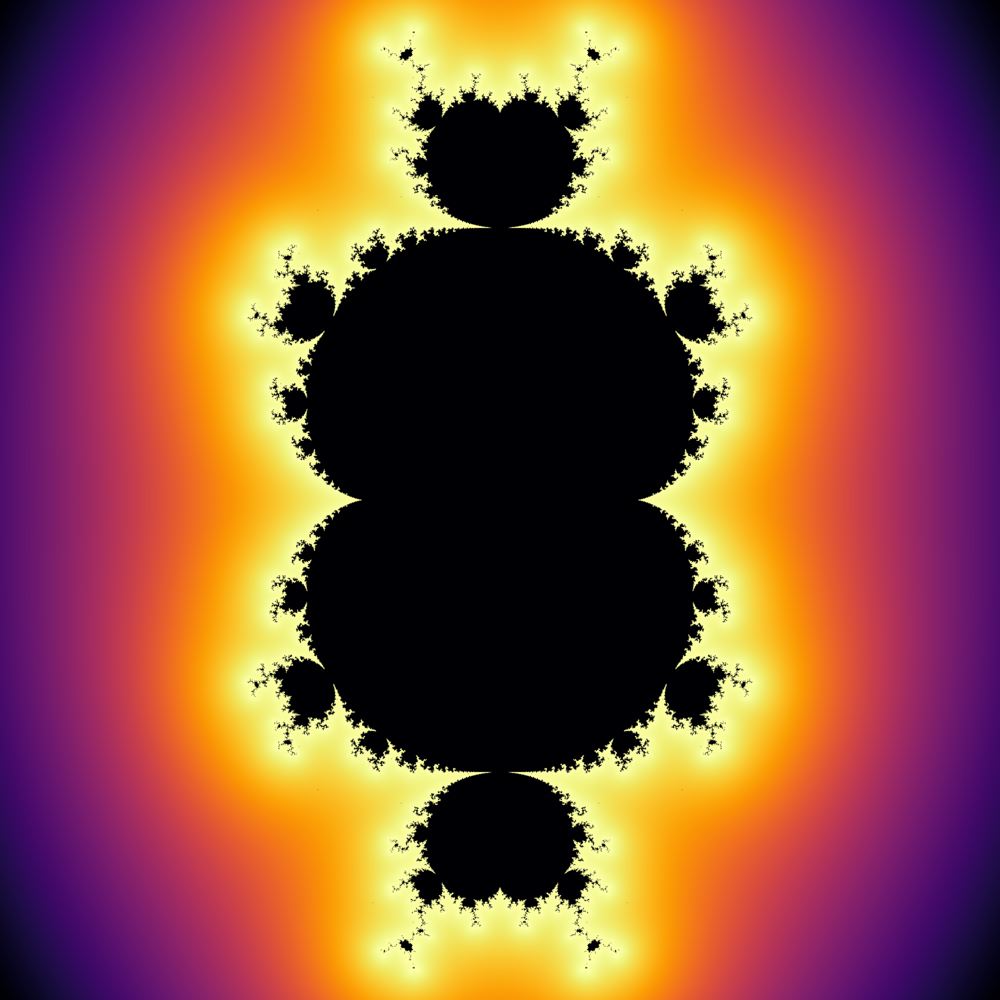}}
		\subfigure[$\cM^8$]{\includegraphics[width=0.3\linewidth]{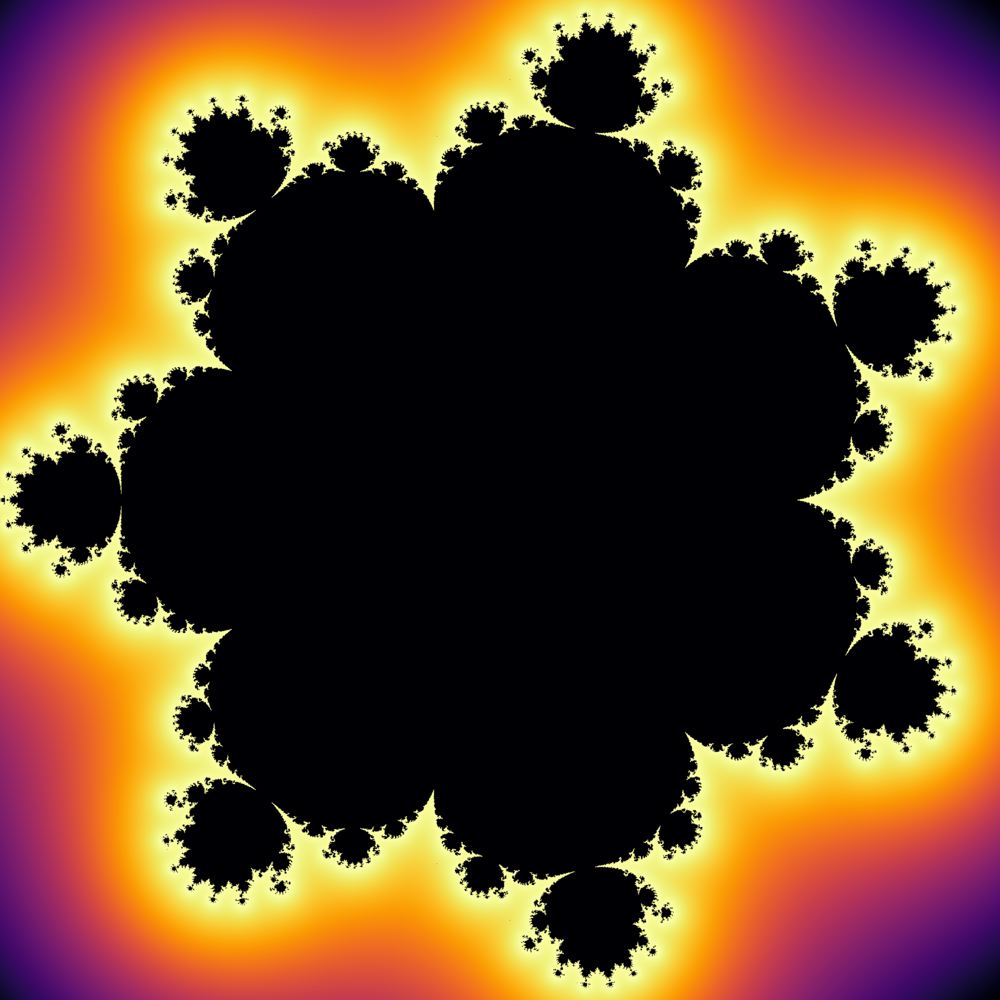}}
		\subfigure[$\cM^{15}$]{\includegraphics[width=0.3\linewidth]{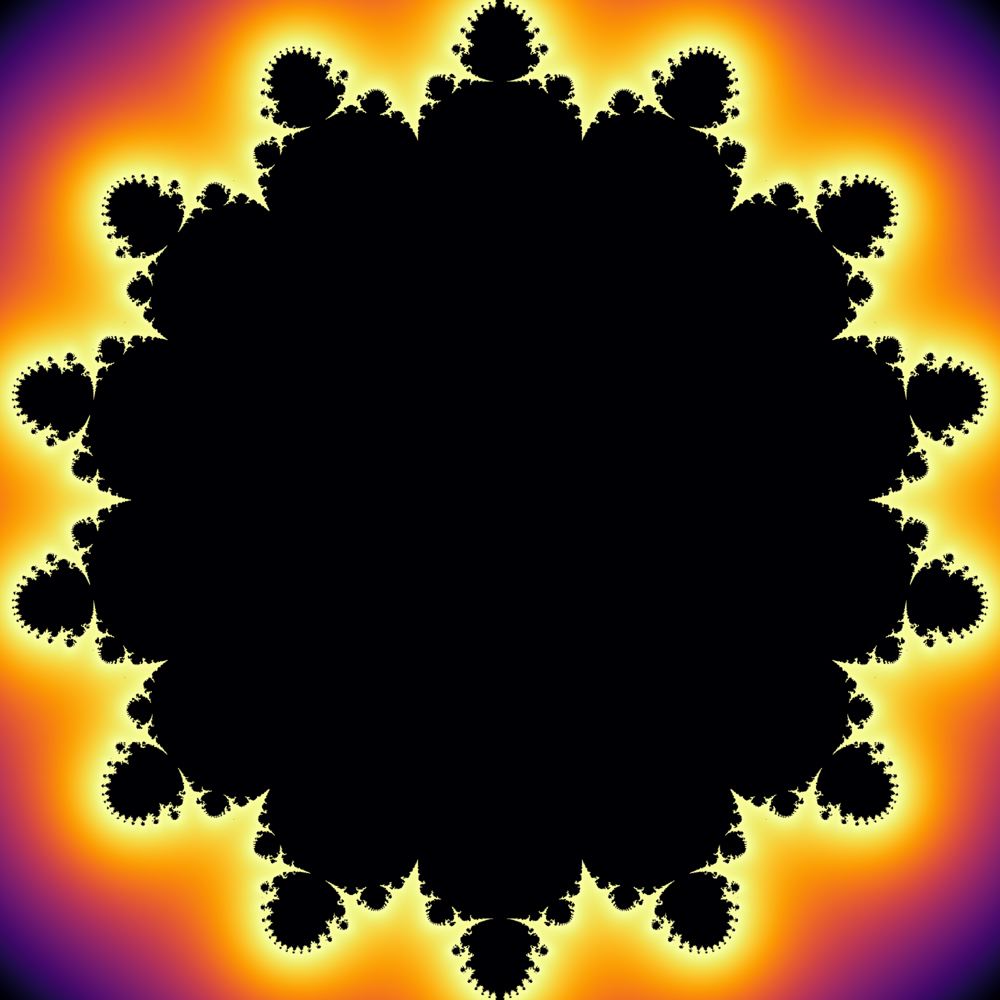}}
		\caption{The Multibrot sets of orders $p=3$, $p=8$ and $p=15$.}
	\end{figure}
		
	\subsection{Distance to a tricomplex fractal}
	The key result is the following. Define, for $\eta \in \TC\backslash\mathrm{X}$,
		\begin{align*}
		d(\eta, \mathrm{X} ) := \inf \oa \Vert \zeta - \eta \Vert_3 \, : \, \zeta \in \mathrm{X} \fa,
		\end{align*}
	which gives the distance between the point $\eta$ and the set $\mathrm{X} \subset \mT \mC$.
	\begin{theorem}\label{KeyTheoremTCDistance}
	If $\mathrm{X} \subset \TC$ is a compact set and 
		\begin{align*}
		\mathrm{X} = (\mathrm{X}_{\id{1}\id{3}} \times_{\id{1}} \mathrm{X}_{\ic{1}\id{3}} ) \times_{\id{3}} (\mathrm{X}_{\id{1}\ic{3}} \times_{\id{1}} \mathrm{X}_{\ic{1}\ic{3}} ),
		\end{align*}
		then
		\begin{align*}
		d( \eta , \mathrm{X} ) = \sqrt{\frac{\sum_{\gamma \in S_{1,3}} d(\eta_{\gamma}, \mathrm{X}_{\gamma})^2 }{4}} \text{.}
		\end{align*}
	\end{theorem}
	
	We will need a lemma to prove this theorem.
	
	\begin{lemma}\label{fContinueSeparable}
	Let $\mathrm{X} \subset \mT \mC$ be a compact set such that
	\begin{align*}
		\mathrm{X} = (\mathrm{X}_{\id{1}\id{3}} \times_{\id{1}} \mathrm{X}_{\ic{1}\id{3}} ) \times_{\id{3}} (\mathrm{X}_{\id{1}\ic{3}} \times_{\id{1}} \mathrm{X}_{\ic{1}\ic{3}} )
	\end{align*}
	and $f : \mathrm{X} \ra [0, +\infty )$ be a continuous function such that
	\begin{align*}
		f(\eta ) = \sum_{\gamma \in S_{1, 3}} f_{\gamma}(\eta_{\gamma})
	\end{align*}
	where $f_{\gamma}$ is continuous $\forall\gamma\in S_{1,3}$. Then, $f$ has a maximum at $a \in \mathrm{X}$ and a minimum at $b \in \mathrm{X}$ with
	\begin{align*}
		a = \sum_{\gamma \in S_{1,3}} a_{\gamma} \gamma \text{ et } b = \sum_{\gamma \in S_{1,3}} b_{\gamma} \gamma
	\end{align*}
	where $a_\gamma,b_\gamma\in\mathrm{X}_\gamma$ are such that $f_\gamma$ attains its maximum and minimum at $a_\gamma$ and $b_\gamma$ respectively.
	\end{lemma}
	\begin{proof}
	Considering that $\mathrm{X}$ is a compact set, it follows that $\mathrm{X}_{\gamma}$ is a compact set $\forall \gamma \in S_{1,3}$. Since $f_\gamma$ is continuous, there exists some $a_{\gamma},b_\gamma \in \mathrm{X}_{\gamma}$ such that $f_\gamma(a_{\gamma}) = \sup \oa f_\gamma(\eta_{\gamma}) \, : \, \eta_{\gamma} \in \mathrm{X}_{\gamma} \fa$, and $f_\gamma(b_{\gamma}) = \inf \oa f_\gamma(\eta_{\gamma}) \, : \, \eta_{\gamma} \in \mathrm{X}_{\gamma} \fa$. Therefore, we see that
	\begin{align*}
		f(\eta ) = \sum_{\gamma \in S_{1,3}} f_{\gamma} (\eta_{\gamma}) \leq \sum_{\gamma \in S_{1,3}} f_{\gamma} (a_{\gamma}) = f(a)\ \forall\eta\in\mathrm{X},
	\end{align*}
	and so
	\begin{align*}
		\sup_{\eta \in \mathrm{X}} f(\eta ) \leq f(a) \text{.}
	\end{align*}
	Since $a \in \mathrm{X}$, it follows that $f(a) = \sup_{\eta \in \mathrm{X}} f(\eta )$. We use similar arguments to prove that $f(b) = \inf_{\eta \in \mathrm{X}} f(\eta )$.\qedhere
	\end{proof}
	
	\begin{proof}[\textit{Proof of Theorem \ref{KeyTheoremTCDistance}}]
	Let $\eta \in \TC\backslash \mathrm{X}$ and $\zeta \in \mathrm{X}$. We know that
		\begin{align*}
		\zeta - \eta = \sum_{\gamma \in S_{1,3}} (\zeta_{\gamma} - \eta_\gamma ) \gamma.
		\end{align*}
	From equation (\ref{NormIdem}) presented in section \ref{SecBasics}, we have that
		\begin{align*}
		\Vert \zeta - \eta \Vert_3^2 = \frac{\sum_{\gamma \in S_{1,3}} |\zeta_\gamma - \eta_\gamma|^2}{4}.
		\end{align*}
	Since $\mathrm{X}$ is a compact set and $f(\zeta) = \Vert \zeta - \eta \Vert_3^2$ satisfies the hypotheses of Lemma \ref{fContinueSeparable}, the function $f$ reaches its minimum at a point $b \in \mathrm{X}$ such that
		\begin{align*}
		b = \sum_{\gamma \in S_{1, 3}} b_{\gamma} \gamma \text{ and } \inf_{\zeta \in \mathrm{X}} \Vert \zeta - \eta \Vert_3^2 = \Vert b - \eta \Vert_3^2 = \frac{\sum_{\gamma \in S_{1,3}} |b_{\gamma} - \eta_\gamma|^2}{4}.
		\end{align*}
	Since $|b_\gamma - \eta_\gamma|^2 = d( \eta_\gamma , \mathrm{X}_{\gamma} )^2$ $\forall \gamma \in S_{1.3}$ from Lemma \ref{fContinueSeparable}, we have that
		\begin{align*}
		d( \eta, \mathrm{X} ) &= \Vert b - \eta \Vert_3 = \sqrt{\frac{\sum_{\gamma \in S_{1,3}} d(\eta_\gamma, \mathrm{X}_{\gamma})^2}{4}}.\qedhere
		\end{align*}
	\end{proof}
	
	By using this theorem with $\mathrm{X} = \cK_{3,c}^p$ or $\mathrm{X} = \cM_3^p$ and by applying the results from the complex plane, namely Theorems \ref{ComplexJDistEstimate} and \ref{ComplexMDistEstimate}, we obtain lower and upper bounds for the distances $d(\eta, \cK_{3,c}^p)$ and $d(c, \cM_3^p)$ respectively.

\section{Computer experiments}
	Theorem \ref{KeyTheoremTCDistance} is the main tool to visualize some tridimensional slices of the tricomplex Julia sets and the tricomplex Multibrot sets using the distance estimations in the complex plane.
	
	Let $\mI(3) = \{1,\im{1},\im{2},\im{3},\im{4},\jm{1},\jm{2},\jm{3}\}$. We now recall some basic definitions from \cite{GarantPelletier, GarantRochon, Parise, RochonParise}.
	
	\begin{definition}
	Let $\im{m},\im{k},\im{l}\in\mI(3)$ with $\im{m} \neq \im{k}$, $\im{m} \neq \im{l}$ and $\im{k} \neq \im{l}$. We define the following vector subspace of $\mT \mC$ : 
		\begin{align*}
		\mT (\im{m} , \im{k} , \im{l} ) := \vspan_{\mR} \oa \im{m} , \im{k} , \im{l} \fa \text{.}
		\end{align*}
	\end{definition}
	\begin{remark}
	The notation $\vspan_{\mR}$ stands for the linear space spanned by some vectors over the field $\mR$. Equivalently, it stands for the space of all linear combinations of those vectors.
	\end{remark}
	
	\begin{definition}
	Let $\im{m},\im{k},\im{l}\in\mI(3)$ with $\im{m} \neq \im{k}$, $\im{m} \neq \im{l}$ and $\im{k} \neq \im{l}$. 
		\begin{enumerate}
		\item We define a 3D principal slice of the Multibrot set $\cM_{3}^p$ as
			\begin{align*}
			\cT^p (\im{m} , \im{k} , \im{l} ) := \oa c \in \mT (\im{m} , \im{k} , \im{l} ) \, : \, \oa Q_{p,c}^m (0) \fa_{m=1}^{\infty} \text{ is bounded} \fa \text{.}
			\end{align*}
		\item We define a 3D principal slice of the filled-in Julia set $\cK_{3,c}^p$ as
			\begin{align*}
			\cF_c^p (\im{m} ,  \im{k} , \im{l} ) := \oa z \in \mT (\im{m} , \im{k} , \im{l} ) \, : \, \oa Q_{p,c}^m (z) \fa_{m = 1}^{\infty} \text{ is bounded} \fa \text{.}
			\end{align*}
		\end{enumerate}
	\end{definition}
	
	\begin{remark}
	When the context is clear, we only write $\cT^p$ instead of $\cT^p (\im{m} , \im{k} , \im{l} )$ and, similarly, $\cF_c^p$ instead of $\cF_c^p (\im{m} , \im{k} , \im{l} )$.
	\end{remark}

	We are now interested in giving some examples of 3D slices of the Multibrot sets and filled-in Julia sets.
	
	\begin{example}
	Let $\im{m} = 1$, $\im{k} = \bo$ and $\im{l} = \bt$. In this case, $\cT^2 (1, \bo , \bt )$ is the classical \textit{Tetrabrot} introduced by Rochon in \cite{Rochon1}. Figure \ref{figTetrabrot} shows some examples of Tetrabrot sets. We note that the colors on the pictures are generated in conformity with the generalized Fatou-Julia Theorem in multicomplex spaces (see \cite{GarantRochon}).
	
	Recall the distance formula from Theorem \ref{KeyTheoremTCDistance},
		\begin{align*}
		d (c, \cM_3^p ) = \sqrt{\frac{\sum_{\gamma \in S_{1,3}} d (c_{\gamma} , \cM^p)^2}{4}}.
		\end{align*}
	So, if $c \in \cT^p (1, \bo , \bt )$, then $c = c_0 + c_1 \bo + c_2 \bt$. From the idempotent representation, we have that
		\begin{align}
		c &= (c_1 + c_2 \bo - c_3\bo ) \id{1} \id{3} + (c_1 + c_2 \bo + c_3 \bo ) \ic{1} \id{3} \notag\\
		&\phantom{=} + (c_1 + c_2 \bo - c_3\bo ) \id{1} \ic{3} + (c_1 + c_2 \bo + c_3 \bo ) \ic{1} \ic{3} \text{.} \label{idempTetra}
		\end{align}
	Thus, let $c_{\id{1}} := c_1 + c_2 \bo - c_3 \bo$ and $c_{\ic{1}} := c_1 + c_2 \bo + c_3 \bo$, then
		\begin{align*}
		d( c , \cM_3^p ) &= \sqrt{\frac{2 d(c_{\id{1}} , \cM^p)^2 + 2 d(c_{\ic{1}} , \cM^p)^2}{4}} = \sqrt{\frac{d(c_{\id{1}} , \cM^p)^2 + d (c_{\ic{1}} , \cM^p )^2}{2}} \text{.}
		\end{align*}
	Using this last equation and the distance approximations in the complex plane, we can trace the slice $\cT^p (1 , \bo , \bt )$ of the Multibrot sets. In fact, knowing that $c\in\cM_3^p\Leftrightarrow c_{\gamma_1},c_{\overline{\gamma}_1}\in\cM^p$, the expression \eqref{idempTetra} leads to the following characterization of the Tetrabrot (see \cite{Rochon1} for the case $p=2$):
		\begin{align*}
		\cT^p (1 , \bo , \bt ) = \bigcup_{y \in \mR} \Big( \left[\left(\cM^p - y\im{1}\right) \cap \left(\cM^p + y\im{1}\right) \right] + y\im{2}\Big).
		\end{align*}
	\begin{figure}
		\centering
		\subfigure[$\cT^3(1,\im{1},\im{2})$]{\includegraphics[scale=0.4]{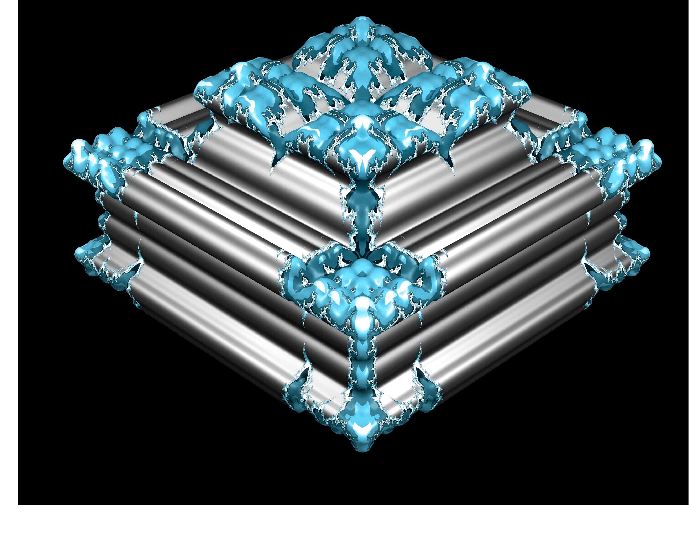}}
		\subfigure[$\cT^4(1,\im{1},\im{2})$]{\includegraphics[scale=0.4]{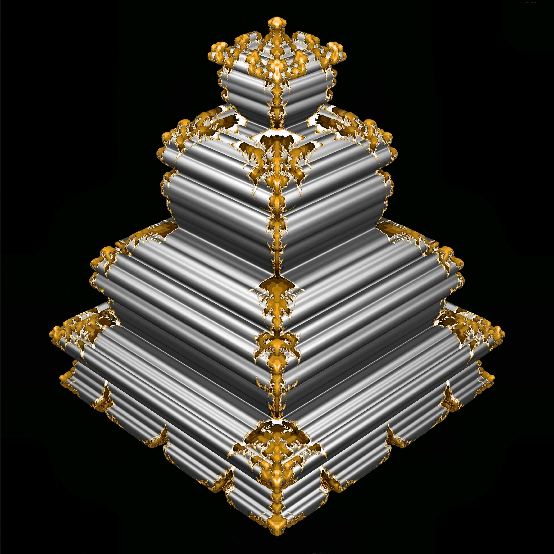}}
		\caption{The Tetrabrot sets with $p=3$ and $p=4$.}\label{figTetrabrot}
	\end{figure}
	\end{example}
	
	\begin{example}
	Let $\im{m} = 1$, $\im{k} = \bjp$, and $\im{l} = \bjd$. Then, the 3D principal slice $\cT^p (1, \bjp , \bjd )$ is called the \textit{Airbrot} (or the \textit{Perplexbrot}). It is proved in \cite{RochonThomParise, RochonParise, RochonParise2} that the Airbrot is a regular octahedron for all integer $p \geq 2$, and using the idempotent representation, we can characterize this 3D slice as 
		\begin{align*}
		\cT^p (1 , \bjp , \bjd ) = \bigcup_{y \in \mR} \Big( \oc \op \cH^p - y\bjp \fp \cap \op \cH^p + y \bjp \fp \fc + y \bjd \Big)
		\end{align*}
	where $\cH^p$ is the hyperbolic Multibrot of order $p$, that is
	\[\cH^p = \{c = x_1 + x_2\jm{1}\ |\ c\in\cM_2^p\text{ and } x_1,x_2\in\mR\}.\]
	See \cite{Parise, RochonThomParise, RochonParise, RochonParise2} for more details on hyperbolic Multibrots.
	
	\begin{figure}
	\centering
		\subfigure[$\cT^2 (1 , \bjp , \bjd )$]{\includegraphics[scale=0.3]{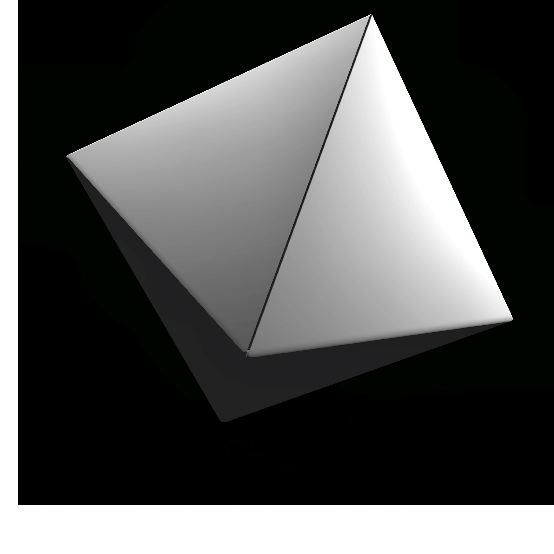}}
		\subfigure[$\cT^5 (1 , \bjp , \bjd )$]{\includegraphics[scale=0.3]{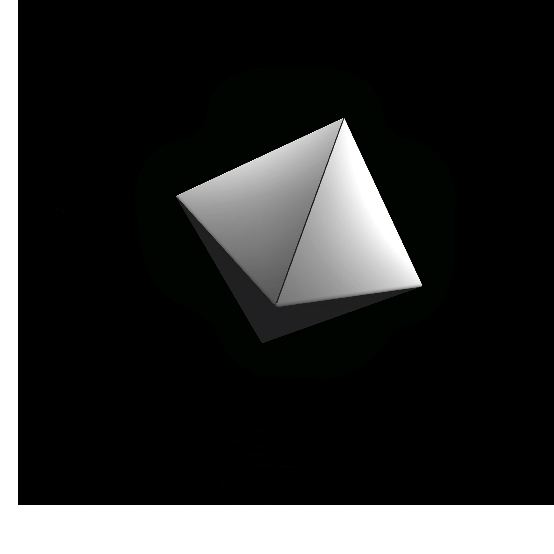}}
		\caption{The Airbrots with $p=2$ and $p=5$.}\label{figPerplexbrot}
	\end{figure}
	\end{example}
	
	\begin{example}
	Let $\im{m} = 1$, $\im{k} = \bo$ and $\im{l} = \bjp$. We get the 3D principal slice $\cT^p (1 , \bo , \bjp )$ called the \textit{Arrowheadbrot}. Figure \ref{figT1i1j1} shows pictures of this slice for $p=3$ and $p=4$. Let $c \in \cT^p (1, \bo , \bjp )$ such that $c = c_1 + c_2 \bo + c_3 \bjp$. Then, the idempotent representation of $c$ is
		\begin{align*}
		c & = (c_1 + c_2 \bo + c_3 ) \id{1} \id{3} + (c_1 + c_2 \bo - c_3 ) \ic{1} \id{3} \notag\\
		& \phantom{ = } + (c_1 + c_2 \bo + c_3 ) \id{1} \ic{3} + (c_1 + c_2 \bo - c_3 ) \ic{1} \ic{3}.
		\end{align*}
	Let $c_{\id{1}} := c_1 + c_2 \bo + c_3$ and $c_{\ic{1}} := c_1 + c_2\bo - c_3$.  From Theorem \ref{KeyTheoremTCDistance}, the distance from $c$ to $\cM_3^p$ is
		\begin{align*}
		d( c , \cM_3^p ) = \sqrt{\frac{d(c_{\id{1}} , \cM^p)^2 + d (c_{\ic{1}}, \cM^p)^2}{2}}.
		\end{align*}
		
		\begin{figure}
		\centering
		\subfigure[$\cT^3 (1 , \bo , \bjp )$]{\includegraphics[scale=0.3]{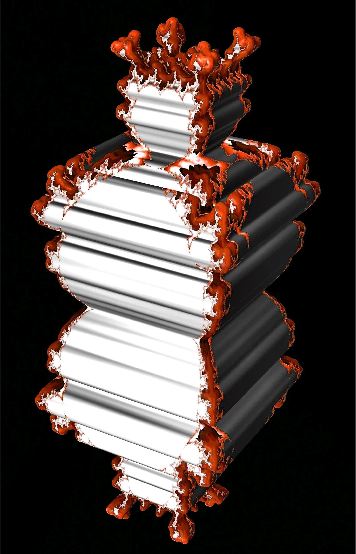}}
		\subfigure[$\cT^4 (1 , \bo , \bjp )$]{\includegraphics[scale=0.3]{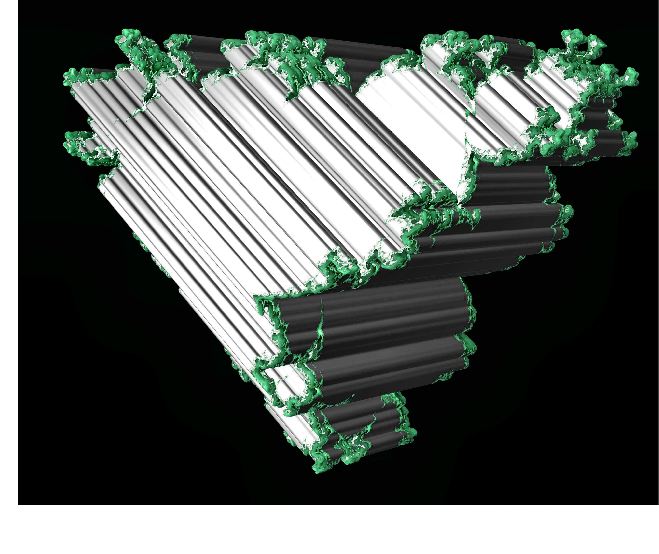}}
		\subfigure[$\cT^3 (1 , \bo , \bjp )$, zoom in]{\includegraphics[scale=0.5]{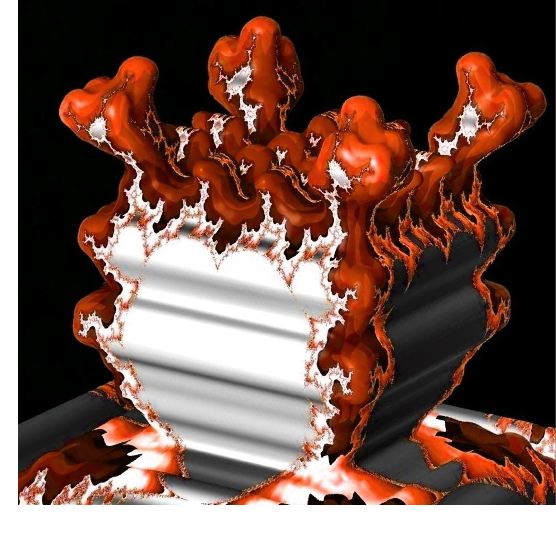}}
		\caption{The Arrowheadbrots with $p=3$ and $p=4$.}\label{figT1i1j1}
		\end{figure}
	\end{example}

\section*{Conclusion}

	Using the tricomplex numbers, we saw that it was possible to generalize the Multibrot and filled-in Julia sets to obtain new hypercomplex 3D fractal images. Moreover, some properties valid in the complex case were extended to the tricomplex space. To be able to generate those new objects, we generalized approximations of the distance between a point outside the fractal and the fractal itself. In the end, it became possible to generate 3D visuals of the 8D geometric objects.
	
	Not all 3D principal slices were shown in this article. Indeed, we focused mainly on the slices of the Multibrot sets rather than the filled-in Julia sets. Moreover, when it comes to a Multibrot set of order $p$, Parisé conjectured that there are at least eight 3D principal slices if $p$ is even and four if $p$ is odd \cite{Parise}.
	
	After generating such images, one could wonder what could be done by generalizing the tricomplex numbers further to the quadricomplex numbers. With sixteen different units, perhaps different 3D slices could be obtained. However, after exploring Multibrot sets beyond the tricomplex space, we strongly believe that there is no other 3D principal slice to find (see \cite{RochonBrouillette}). Nonetheless, the Julia sets could still surprise us.
	
	On another note, it would be possible to define different types of 3D slices. For example, the idempotent numbers give us another way to represent tricomplex numbers and, consequently, another way to define 3D slices. Therefore, the idempotent slices of the Multibrot and filled-in Julia sets will be another topic to address subsequently.

\section*{Acknowledgement}
 
DR is grateful to the Natural Sciences and Engineering Research Council of Canada (NSERC) for financial support. GB would like to thank the Fonds de Recherche du Québec - Nature et technologies (FRQNT) and the Institut des sciences mathématiques (ISM) for the awards of graduate research grants. POP would also like to thank the NSERC for the awards of a graduate research grant. The authors are grateful to Louis Hamel and Étienne Beaulac, from UQTR, for their useful work on the MetatronBrot Explorer in Java.

\bibliographystyle{abbrv}
\bibliography{ArticleDistance_Final}

\begin{thebibliography}{10}

\bibitem{Beardon}
A.~F. Beardon.
\newblock {\em Iteration of Rational Functions: Complex Analytic Dynamical
  Systems}, volume 132 of {\em Graduate Texts in Mathematics}.
\newblock Springer-Verlag New York, first edition, 1991.

\bibitem{RochonBrouillette}
G.~Brouillette and D.~Rochon.
\newblock Characterization of the principal {3D} slices related to the
  multicomplex {M}andelbrot set.
\newblock {\em Advances in Applied Clifford Algebras}, 29(3), 2019.

\bibitem{Conway}
J.~B. Conway.
\newblock {\em Functions of One Complex Variable {II}}, volume 159 of {\em
  Graduate Texts in Mathematics}.
\newblock Springer-Verlag New York, first edition, 1995.

\bibitem{Dang}
Y.~Dang, L.~H. Kauffman, and D.~J. Sandin.
\newblock {\em Hypercomplex Iterations: Distance Estimation and Higher
  Dimensional Fractals}.
\newblock World Scientific, 2002.

\bibitem{Douady}
A.~Douady and J.~H. Hubbard.
\newblock It{\'e}ration des polyn{\^o}mes quadratiques complexes.
\newblock {\em C.R. Acad. Sci. Paris -- S{\'e}rie {I}, Math.}, 294:123--126,
  1982.

\bibitem{GarantPelletier}
V.~Garant-Pelletier.
\newblock Ensembles de {Mandelbrot} et de {Julia} classiques,
  g{\'e}n{\'e}ralis{\'e}s aux espaces multicomplexes et th{\'e}or{\`e}me de
  {Fatou-Julia} g{\'e}n{\'e}ralis{\'e}.
\newblock Master's thesis, Universit{\'e} du Qu{\'e}bec {\`a}
  {Trois-Rivi{\`e}res}, Canada, 2011.

\bibitem{GarantRochon}
V.~Garant-Pelletier and D.~Rochon.
\newblock On a generalized {Fatou-Julia} theorem in multicomplex spaces.
\newblock {\em Fractals}, 17(3):241--255, 2009.

\bibitem{Luna}
M.~E. Luna-Elizarrar{\'a}s, M.~Shapiro, D.~C. Struppa, and A.~Vajiac.
\newblock {\em Bicomplex Holomorphic Functions: The Algebra, Geometry and
  Analysis of Bicomplex Numbers}.
\newblock Frontiers in Mathematics. Birkh{\"a}user, 2015.

\bibitem{Martineau}
{\'E}.~Martineau.
\newblock Bornes de la distance à l'ensemble de {Mandelbrot}
  g{\'e}n{\'e}ralis{\'e}.
\newblock Master's thesis, Universit{\'e} du Qu{\'e}bec {\`a}
  {Trois-Rivi{\`e}res}, Canada, 2004.

\bibitem{RochonMartineau}
{\'E}.~Martineau and D.~Rochon.
\newblock On a bicomplex distance estimation for the {Tetrabrot}.
\newblock {\em International Journal of Bifurcation and Chaos},
  15(9):3039--3050, 2005.

\bibitem{Milnor1}
J.~Milnor.
\newblock Arithmetic of unicritical polynomial maps.
\newblock {\em Frontiers in Complex Dynamics: In Celebration of John Milnor’s
  80th Birthday}, pages 15--23.

\bibitem{Milnor2}
J.~Milnor.
\newblock {\em Dynamics in One Complex Variable : Introductory Lectures}.
\newblock Springer Vieweg, second edition, 2000.

\bibitem{Parise}
P.-O. Paris{\'e}.
\newblock Les ensembles de {Mandelbrot} tricomplexes g{\'e}n{\'e}ralis{\'e}s
  aux polyn{\^o}mes $\zeta^p+c$.
\newblock Master's thesis, Universit{\'e} du Qu{\'e}bec {\`a}
  {Trois-Rivi{\`e}res}, Canada, 2017.

\bibitem{RochonThomParise}
P.-O. Paris{\'e}, T.~Ransford, and D.~Rochon.
\newblock Tricomplex dynamical systems generated by polynomials of even degree.
\newblock {\em Chaotic Modeling and Simulation (CMSIM)}, 1:37--48, 2017.

\bibitem{RochonParise}
P.-O. Paris{\'e} and D.~Rochon.
\newblock A study of dynamics of the tricomplex polynomial $\eta^p+ c$.
\newblock {\em Nonlinear Dynamics}, 82(1--2):157--171, 2015.

\bibitem{RochonParise2}
P.-O. Paris{\'e} and D.~Rochon.
\newblock Tricomplex dynamical systems generated by polynomials of odd degree.
\newblock {\em Fractals}, 25(3):1--11, 2017.

\bibitem{Peitgen}
H.-O. Peitgen and D.~Saupe, editors.
\newblock {\em The Science of Fractal Images}.
\newblock Springer-Verlag New York, Inc., first edition, 1988.

\bibitem{Baley}
G.~B. Price.
\newblock {\em An Introduction to Multicomplex Spaces and Functions}.
\newblock M. Dekker, 1991.

\bibitem{Rochon1}
D.~Rochon.
\newblock A generalized {Mandelbrot} set for bicomplex numbers.
\newblock {\em Fractals}, 8(4):355--368, 2000.

\bibitem{Rochon3}
D.~Rochon.
\newblock A {Bloch} constant for hyperholomorphic functions.
\newblock {\em Complex Variables}, 44:85--201, 2001.

\bibitem{Rochon2}
D.~Rochon.
\newblock On a generalized {Fatou}-{Julia} theorem.
\newblock {\em Fractals}, 11(3):213--219, 2003.

\bibitem{RochonShapiro}
D.~Rochon and M.~Shapiro.
\newblock On algebraic properties of bicomplex and hyperbolic numbers.
\newblock {\em Anal. Univ. Oradea, fasc. math}, 11(71):110, 2004.

\bibitem{Rudin}
W.~Rudin.
\newblock {\em Real and Complex Analysis}.
\newblock McGraw-Hill, Inc., third edition, 1987.

\bibitem{vajiac2}
M.~Shapiro, D.~C. Struppa, A.~Vajiac, and M.~B. Vajiac.
\newblock Hyperbolic numbers and their functions.
\newblock {\em Anal. Univ. Oradea}, XIX(1):265--283, 2012.

\bibitem{Sobczyk}
G.~Sobczyk.
\newblock The hyperbolic number plane.
\newblock {\em The College Mathematics Journal}, 26(4):268--280, 1995.

\bibitem{vajiac}
A.~Vajiac and M.~B. Vajiac.
\newblock Multicomplex hyperfunctions.
\newblock {\em Complex Variables and Elliptic Equations}, 57(7--8):751--762,
  2012.

\bibitem{Wang}
X.-y. Wang and W.-j. Song.
\newblock The generalized {M}--{J} sets for bicomplex numbers.
\newblock {\em Nonlinear Dynamics}, 72(1--2):17--26, 2013.

\end{thebibliography}

\end{document}